\documentclass[10pt]{amsart}
\usepackage{graphicx}
\usepackage{amsfonts}
\usepackage{eucal}
\usepackage{amssymb}
\usepackage{amsthm}
\usepackage{amscd}
\usepackage{multicol}
\usepackage{verbatim}
\usepackage{amsmath, amsthm, mathrsfs, hyperref}
\hypersetup{linkcolor=black, citecolor=black, colorlinks=true,
pdfauthor={Matthew Chasse}} 

\theoremstyle{plain}
\newtheorem{theorem}{Theorem}[section]

\newtheorem{lemma}[theorem]{Lemma}

\theoremstyle{definition}
\newtheorem{defn}[theorem]{Definition}
\newtheorem{conj}[theorem]{Conjecture}

\newtheorem{definition}[theorem]{Definition}
\newtheorem{conjecture}[theorem]{Conjecture}

\newtheorem{question}[theorem]{Question}

\theoremstyle{remark}
\newtheorem{remark}[theorem]{Remark}

\def\rl{\mathbb{R}}

\def\nn{\mathbb{N}}
\def\RR{\mathbb{R}}

\def\NN{\mathbb{N}}
\newcommand{\bea}{\begin{eqnarray*}} 
\newcommand{\eea}{\end{eqnarray*}}  
\newcommand{\bnum}{\begin{enumerate}}
\newcommand{\enum}{\end{enumerate}}
\newcommand{\bit}{\begin{itemize}}  
\newcommand{\eit}{\end{itemize}}
\newcommand{\beq}{\begin{equation}}  
\newcommand{\eeq}{\end{equation}}

\begin{document} 
\title[Ultraspherical basis transformation]{Monomial to ultraspherical basis transformation and the zeros of polynomials}

\author[M.~Chasse]{Matthew Chasse}
\address{Department of Mathematics, The Royal Institute of Technology, SE-100 44 Stockholm, Sweden}
\email{chasse@math.hawaii.edu}


\begin{abstract}
We examine a result of A.~Iserles and E.~B.~Saff, use it to prove a conjecture of S.~Fisk that a linear operator which maps monomials to Legendre polynomials also preserves zeros in the interval $(-1,1)$, and state a more general version of the conjecture for the Jacobi polynomials.  
\end{abstract}

\maketitle

\section{Introduction}
   
Let $P_n^{(\alpha,\beta)}(x)$ represent the $n$-th Jacobi polynomial with parameters $\alpha$ and $\beta$, defined by the generating function \cite[p. 271]{R},

\begin{equation}\label{jacobi-gen} 
F(x,t) := \frac{2^{\alpha+\beta}}{\rho(1+t+\rho)^{\beta}(1-t+\rho)^{\alpha}} = \sum_{k=0}^\infty P_{n}^{(\alpha,\beta)}(x) t^n ,
\end{equation}
\\
where $\rho = (1-2xt+t^2)^{1/2}$.
The set $\{P_n^{(\alpha,\alpha)}(x)\}_{n=0}^\infty$, known as the ultraspherical polynomials \cite[p. 276]{R}, is also obtained from the simpler generating function
\\
\begin{equation}\label{Pgen}
G(x,t):=(1-2xt+t^2)^{-\alpha-\frac{1}{2}} = \sum_{n=0}^\infty \frac{(1+2\alpha)_n}{(1+\alpha)_n}P^{(\alpha,\alpha)}_n(x)t^n,
\end{equation}
\\
where $(x)_n = x(x+1)(x+2)\dots(x+n-1)$ is the commonly used \emph{Pochhammer symbol}.  When $\alpha=0$, the ultraspherical polynomials are the classical Legendre polynomials, which we denote by $P_k(x)$ ($P_k(x)=P^{(0,0)}_k(x)$).
We prove the following conjecture from S. Fisk's electronic book \cite{fisk}, which is left as an open problem in the book's appendices.
\begin{conj}{\rm (S. Fisk \cite[p. 724, Question 39]{fisk})}\label{Pconj}
If $f(x)=\sum_{k=0}^n a_k x^k\in\rl[x]$ has all of its zeros in the interval $(-1,1)$ then $T[f](x)=\sum_{k=0}^n a_k P_k(x)$ has all of its zeros in $(-1,1)$.
\end{conj}
The proof of Conjecture \ref{Pconj} is accomplished through a technique of A.~Iserles and E.~B.~Saff.  Many similar transformations involving orthogonal bases have already been established by Iserles, N{\o}rsett, and Saff using the same method in conjunction with other identities \cite{IN1, IN2, IS}.  Our main result is the following.

\begin{theorem}\label{main}
If $f(x)=\sum_{k=0}^n a_k x^k\in\rl[x]$ has all of its zeros in the interval $(-1,1)$ then 
\[
T[f](x)=\sum_{k=0}^n a_k \frac{k!}{\Gamma(k+1+\alpha)} P^{(\alpha,\alpha)}_k(x)
\]
 has all of its zeros in $(-1,1)$ for $\alpha>-1$.
\end{theorem}

Conjecture \ref{Pconj} follows from Theorem \ref{main} with $\alpha=0$.  Iserles and Saff already proved that $x^n\to  (n!/\Gamma(n+1+\alpha)) (2n+2\alpha+1)P^{(\alpha,\alpha)}_n(x)$ maps polynomials whose zeros lie in the interval $(-1,1)$ to polynomials whose zeros lie in the interval $[-1,1]$ \cite[Proposition 3]{IS}.  We modify the generating function \eqref{Pgen} with a differential operator and apply the observation in Lemma \ref{open-sets} (below) to change the Iserles--Saff proof into the one we require. 

We provide a short proof of the Iserles--Saff theorem in Appendix \ref{biortho-polys}, both for reference and to aid in the discussion, as we state it in a more general form.  The characterization of all linear transformations which preserve the set of polynomials whose zeros lie on an interval is an open problem as mentioned in \cite{IN2}, and stated more formally in \cite{BB1, CC3}.  A successful characterization would hopefully provide simple proofs of Theorem \ref{main} and Theorem \ref{key}. 

Total positivity and sign regularity are introduced in Section \ref{section-tp}, along with Theorem \ref{key} of Iserles and Saff.  We also discuss problems which appear difficult with the currently available techniques for establishing sign regularity.  
In Section \ref{section-proof}, we prove Theorem \ref{main} and pose a conjecture which extends it (Conjecture \ref{jacobi-trans}).

{\bf Acknowledgements.}  The author thanks Professor P.~Br\"and\'en for his comments and corrections, and Professor G.~Csordas for suggesting the application of the Iserles--Saff theorem to some of Fisk's conjectures.

\section{Total positivity and the Iserles-N{\o}rsett-Saff theorem}\label{section-tp}

In the proof of Conjecture \ref{Pconj}, we rely on the strict sign regularity (Definition \ref{def:ssr}) of a generating function related to the transformation.   

\begin{defn}\label{def:ssr}
The kernel function $G:X\times Y\to\rl$ is said to be \emph{strictly sign regular} (SSR) if, given any two increasing m-tuples  $x_1 < x_2 < \ldots < x_m$,  $y_1 < y_2 < \ldots < y_m$, in $X$ and $Y$ respectively, 

\beq
\varepsilon(m)\left| \begin{array}{cccc}
G(x_1,y_1) &  G(x_1,y_2) & \ldots & G(x_1,y_m) \\
G(x_2,y_1) & G(x_2,y_2)  & \ldots & G(x_2,y_m)\\
\vdots & \vdots & \ddots & \vdots \\
G(x_m,y_1)  & G(x_m,y_2)  & \ldots & G(x_m,y_m) \end{array} \right|>0.
\label{detcond}
\eeq
for all $m\in\nn$, where $\varepsilon(m)$ is a function of $m$ only and has the range $\{-1,1\}$.  If $\varepsilon(m)=1$, the kernel $G(x,t)$ is said to be \emph{strictly totally positive} (STP), and if equality is allowed \emph{totally positive} (TP)  (see \cite{karlin}). 
\end{defn}

A sequence of functions $\phi_0(x), \phi_1(x),\dots, \phi_n(x)$ is said to constitute a \emph{Tchebyshev system} on $a<x<b$ if for any set of real constants $\{c_k\}$, not all zero, $\sum_{k=0}^n c_k\phi_k(x)$ does not vanish more than $n$ times on $a<x<b$ \cite[p.~24]{karlin}.  If every subset of a Tchebyshev system is also a Tchebyshev system, then $\phi_y(x)$ is SSR.  

\begin{theorem}[{\cite[p. 18]{karlin}}]\label{factortheorem}
Let $K(x,y)$ be $SSR$ \textup{(}$x\in X$ and $y\in Y$\textup{)}, and suppose $\varphi(x)$ and $\psi(x)$ maintain the same constant sign on $X$ and $Y$, respectively.  Then
\bnum
\item $L(x,y)=\varphi(x)\psi(y)K(x,y)$ is $SSR$.

\item Now suppose $u=\varphi^{-1}(x)$ and $v=\psi^{-1}(x)$ be strictly increasing functions mapping $X$ and $Y$ onto $U$ and $V$, respectively, where $\varphi^{-1}$ and $\psi^{-1}$ are the inverse functions of $\varphi$ and $\psi$, respectively.

Consider
$$ L(u,v) = K[\varphi(u),\psi(x)] \text{ \hspace{ 5mm} } u\in U, v\in V.$$
Then $L(u,v)$ is $SSR$ and with signs $\varepsilon_m(K) = \varepsilon_m(L)$.  If $\phi(u)$ is strictly increasing while $\psi(v)$ is strictly decreasing, then $L(u,v)$ is $SSR$ and $\varepsilon_m(K) = (-1)^{\frac{m(m-1)}{2}}\varepsilon_m(L)$.
\enum
\end{theorem}

\begin{theorem}[Composition Rule]\label{cbtheorem}
Let $K$, $L$ and $M$ be Borel-measurable functions of two variables satisfying
\[
M(x,y) = \int_Z K(x,\eta)L(\eta,y) d\sigma(\eta),
\]
where the integral is assumed to converge absolutely.  Suppose the variables $x$, $y$, and $\eta$ are in subsets of the real line $X$, $Y$, and $Z$ respectively, and $d\sigma(\eta)$ is a sigma finite measure on $Z$.  If both $K$ and $L$ are SSR (STP), then so is $M$ \cite[p. 16]{karlin}. 
\end{theorem}

\begin{theorem}[{\cite[p. 15]{karlin}}]\label{exy}
$$K(x, y)=e^{xy}$$
is STP for $x,y\in\rl$.
\end{theorem}
It is worth noting that a generalization of Theorem \ref{exy} exists for power series in the form $\sum_{k=0}^\infty a_kx^ky^k$ \cite[Part V, \#86]{PSz2}.

Let $\RR^n/S_n$ be $n$-dimensional Euclidean space modulo permutations of the coordinates, e.g. $\RR^2/S_2=\{ \{(x_1,x_2),(x_2,x_1)\} : x_1, x_2\in\RR \}$.  Then, $\RR^n/S_n$ is a metric space when endowed with a Hausdorff distance between equivalence classes, 
\[
d(y_1,y_2)= \inf_{a\in y_1, b\in y_2} ||a-b||_{\infty},
\]
for $y_1,y_2\in\RR^2/S_2$, where $||(x_1,\dots,x_n)||_{\infty}=\max\{|x_1|,\dots,|x_n|\}$.  This provides a natural topology for the zeros of polynomials of a fixed degree.  Let $I$ be an open or closed interval, and let $I^n/S_n$ denote the subset of $\RR^n/S_n$ where every component lies in the interval $I$.  The topology on $\RR^n/S_n$ induced by the metric $d(\cdot,\cdot)$ then has open sets which are generated by unions of the ``open intervals'' $(a,b)^n/S_n$, $a,b\in\RR$, $a\le b$.

\begin{lemma}\label{open-sets}
Suppose $T:\RR[x]\to\RR[x]$ is degree preserving, linear, invertible, and maps polynomials whose zeros lie in the interval $(c,d)$ to polynomials whose zeros lie in the interval $[a,b]$. Then $T$ furthermore maps polynomials whose zeros lie in $(c,d)$ to polynomials whose zeros lie in $(a,b)$.
\end{lemma}

\begin{proof}
Let $Q(x)=\sum_{k=0}^nq_kx^k$ be an arbitrary polynomial with zeros in $(c,d)$. 
Let the map $\phi:(c,d)^n/S_n\to[a,b]^n/S_n$ be defined by the mapping of the zeros of $Q$ to the zeros of $T[Q]$.  As $T$ is degree preserving and invertible, both $\phi$ and $\phi^{-1}$ are well-defined, and both are continuous by linearity and Hurwitz's theorem.  Therefore, open sets in $(c,d)^n/S_n$ are mapped by $\phi$ to open sets in $[a,b]^n/S_n$, whence the zeros of $T[Q(x)]$ must lie in $(a,b)$, which covers all open sets contained in $[a,b]^n/S_n$.
\end{proof}

Theorem \ref{key} of A.~Iserles and E.~B.~Saff is required for our proof of Conjecture \ref{Pconj}, and a short proof is provided for the reader's convenience in Appendix \ref{biortho-polys}.  The original version of Theorem \ref{key} is sharpened here by using the observation in Lemma \ref{open-sets}, and relaxing the domain basis to a complete Tchebyshev system.

\begin{theorem}[A. Iserles and E. B. Saff \cite{IS}]\label{key}
Let $\{Q_n(x)\}_{n=0}^\infty$ be an orthogonal sequence of polynomials satisfying 
\[
\int_{a}^b Q_n(x)Q_m(x) dx = \begin{cases}h_m , & \text{ for } m=n\cr
0, & \text{ for } m\neq n.
\end{cases}
\]
Let $\{R_k\}_{k=0}^\infty$ be a set of polynomials, with $\deg(R_k)=k$ (thus any subset  $\{R_k\}_{k=0}^n$, $n\in\NN$ is a Tchebyshev system on $(c,d)$).  
If the generating function $G(x,t)=\sum_{t=0}^\infty \delta_kQ_k(x)R_k(t)$ is SSR for all $x\in(a,b)$ and $t\in(c,d)$, and all the zeros of the polynomial $\sum_{k=0}^nq_kR_k(x)$ are in the interval $(c,d)$, then all the zeros of the polynomial 
\[
\sum_{k=0}^n\frac{q_k}{\delta_kh_k}Q_k(x)
\]
are in the interval $(a,b)$.
\end{theorem} 

All of the transformations which have been obtained so far with Theorem \ref{key} have $R_k(t)=t^k$.  For other starting bases, showing the kernel
\[
G(x,t)=\sum_{t=0}^\infty \delta_kQ_k(x)R_k(t)
\]
is SSR on some region appears quite difficult. Two papers by J.~Burbea \cite{burbea2,burbea} establish the strict total positivity of several replicating kernels by extending the notion of total positivity and sign regularity to curves in the complex plane, but this approach is very restrictive.  For natural extensions of total positivity to the complex plane, compositions of the type in Theorem \ref{cbtheorem} are generally not possible without unusually convenient phase behavior of the kernels involved.  
The lack of tools for determining totally positivity and sign regularity of kernels prevents application of Theorem \ref{key} in some desirable situations.
For example, in \cite{IS}, a comment is made that although the kernel
\[
\frac{1}{(1-2xy+y^2)^{\beta}}
\]
is not strictly totally positive on $(-1,1)\times (-1,1)$ for $\beta<0$, it appears to be strictly sign regular there. 
Similarly, if the range of parameters for which \eqref{jacobi-gen} is STP or SSR could be determined, a large number of transformations involving the Jacobi polynomials would follow.

\section{Proof for the ultraspherical polynomial transform}\label{section-proof}

Recall that for $\Re(\alpha)>-1$, $\Re(\beta)>-1$, the Jacobi polynomials satisfy the orthogonality relation \cite[p. 260]{R}

\begin{align}
&\int_{-1}^1 P^{(\alpha,\beta)}_n(x)P^{(\alpha,\beta)}_m(x) (1-x)^\alpha(1+x)^\beta dx = \label{ortho} \\
&\qquad \begin{cases} \displaystyle \frac{2^{1+\alpha+\beta}\Gamma(1+\alpha+n)\Gamma(1+\beta+n)}{(2n+1+\alpha+\beta)n!\Gamma(1+\alpha+\beta+n)}, & \text{ for } m=n\\
\\
 \displaystyle 0, & \text{ for } m\neq n.\nonumber
\end{cases} 
\end{align}
From Theorems \ref{cbtheorem} and \ref{exy} one obtains the next lemma.

\begin{lemma}[\cite{IS}]\label{handy}
The function
\[
K(x,y) = \frac{1}{(x+y)^{\beta}},
\]
where $\beta>0$, is STP for $x,y\in(0,\infty)$ 
\end{lemma}

By simple substitutions with Theorem \ref{factortheorem}, see \cite{IS}, Lemma \ref{handy} is transformed into what we require.

\begin{lemma}\label{legendre-gen-stp}
The function
\[
K(x,y) = \frac{1}{(1-2xy+y^2)^{\beta}},
\]
where $\beta>0$, is STP for $x,y\in(-1,1)$. 
\end{lemma}

\begin{proof}[Proof of Theorem \ref{main}]
The series on the right-hand side of  \eqref{Pgen} is known to converge for all $x,t\in(-1,1)$.
Operating on \eqref{Pgen} with $\left(2t\frac{\partial}{\partial t} + 2\alpha + 1 \right)$, we obtain a second generating function which converges for the same range of parameters:
\begin{align*}
G_2(x, t) =& \left(2t\frac{\partial}{\partial t} + 2\alpha + 1 \right)G(x,t)\\
=& \frac{(2\alpha+1)(1-t^2)}{(1-2xt+t^2)^{\alpha+3/2}} \\
=& \sum_{k=0}^\infty (2k+\alpha+1)\frac{(1+2\alpha)_k}{(1+\alpha)_k}P^{(\alpha,\alpha)}_k(x)t^k.
\end{align*}
\noindent
By Theorem \ref{key}, with $\delta_k=2k+2\alpha+1$, and (from \eqref{ortho} with $\alpha=\beta> -1$)
\[
h_k=\frac{2^{1+\alpha}(\Gamma(1+\alpha+n))^2}{k!(2k+2\alpha+1)\Gamma(1+2\alpha+2k)},
\] 
it is sufficient to show that $G_2(x,t)$ is SSR for $x,t\in(-1,1)$.  By Theorem \ref{factortheorem}, the factor $(2\alpha+1)(1-t^2)$ can be dropped, and it is sufficient to check that
\[
 \frac{1}{(1-2xt+t^2)^{\alpha+3/2}} 
\]
is SSR for $\alpha>-1$, which follows by Lemma \ref{legendre-gen-stp}. 
\end{proof}

It is natural to question whether Theorem \ref{main} can be extended to the Jacobi Polynomials. Fisk has posed the following related question.

\begin{question}[{\cite[p. 724, Question 40]{fisk}}]\label{jacobi-fact}
For which $\alpha$ and $\beta$ does the transformation $x^n\to P^{(\alpha,\beta)}_n(x)/n!$ preserve the set of polynomials with only real zeros? 
\end{question}

Numerical experiments using polynomials of the form $(x-1)^n(x+1)^m$, $n+m\le14$, suggest the restriction $\alpha,\beta\ge0$ in Question \ref{jacobi-fact}.  As an extension of Conjecture \ref{Pconj} (settled by Theorem \ref{main}) we state the following conjecture, which has been verified using Mathematica for $\alpha,\beta<5$, with $(x-1)^n(x+1)^m$, $n+m\le14$, and a large number of random polynomials of degree less than ten with zeros in the interval $(-1,1)$.

\begin{conjecture}\label{jacobi-trans}
If $\alpha$ and $\beta$ are non-negative integers, then $x^n\to P^{(\alpha,\beta)}_n(x)$ preserves the set of polynomials whose zeros lie only in the interval $(-1,1)$.
\end{conjecture}

\renewcommand\refname{References}

\newpage 

\appendix

\section{Biorthogonal polynomials}\label{biortho-polys}

\begin{definition}
Let $(a,b)\subseteq\RR$, and let the sequence $\{t_k\}_{k=1}^\infty\subset\Omega\subset\RR$ be such that $t_k\neq t_\ell$ when $k\neq\ell$.  Define a bivariate distribution $\phi:(a,b)\times\Omega\to [0,\infty)$.  A \emph{biorthogonal polynomial system} with respect to $\phi$ is a set of polynomials $\{p_m\}_{m=0}^\infty$ such that $\deg(p_m)=m$, and 
\beq
\int_a^b p_m(x) d\phi(x,t_\ell) = 0 \qquad 1\le\ell\le m. 
\eeq
A member $p_m$ said to be a \emph{biorthogonal polynomial}.
\end{definition}

Note that biothogonal polynomials are only orthogonal to the specified subset of measure functions and not necessarily with respect to each other.

\begin{theorem}[{\cite{IN1}}]
Let 
\[
I_k(t):=\int_a^bx^kd\phi(x,t), \qquad k\ge0. 
\]
There exists a unique biorthogonal polynomial system with respect to $\phi$ if and only if 
\[
D_m(t_1,t_2,\ldots,t_m):=
\begin{vmatrix}
I_0(t_1) & I_1(t_1) & \cdots & I_{m-1}(t_1)\\
\vdots     &            &        & \\
I_0(t_m) & I_1(t_m) & \cdots & I_{m-1}(t_m)
\end{vmatrix}
\neq 0.
\]
\end{theorem}

A weight function $\phi(x,t)$ which satisfies $D_m(t_1,t_2,\ldots,t_m)\neq 0$ is said to be \emph{regular}.
Assume that $d\phi(x,t) = \omega(x,t)d\alpha(x)$, where $\alpha(x)$ is a distribution independent of $t$.  The weight function $\omega$ is said to possess the \emph{interpolation property}, if for $m\ge 1$, $\{t_k\}_{k=1}^\infty\subset\Omega\subset\RR$, distinct $x_1,x_2,\ldots,x_m\in(a,b)$,  and $y_1,y_2,\ldots,y_m\in\RR$, there exist real constants $\beta_1, \beta_2, \ldots, \beta_m$ such that 
\[
\sum_{\ell=1}^m \beta_\ell\omega(x_k,t_\ell) = y_k \qquad 1\le k\le m
\]

\begin{theorem}[{\cite{IN1}}]
The distribution  $\omega(x,t)$ possesses the interpolation property if and only if 
\[
\det\left[\omega(x_k,t_\ell) \right]_{j=1,k=1}^{j=m,k=m}\neq 0,
\]
for all $m\in\NN$, and all possible sequences $x_1,x_2,\ldots\in(a,b)$ and $t_1,t_2,\ldots\in\Omega\subset\RR$.
\end{theorem}

\begin{remark}
Clearly, if $\omega(x,t)$ is totally positive or strictly sign regular on $(a,b)\times\Omega$ (see Definition \ref{def:ssr}) then it has the interpolation property on $(a,b)\times\Omega$ as well.
\end{remark}

\begin{theorem}[{\cite{IN1}}]\label{real-zeros}
If $\omega(x,\cdot)\in C_1(t)$, possesses the interpolation property, and corresponds to a regular distribution, then each corresponding biorthogonal polynomial $p_m(x)= p_m(x,t_1,\ldots,t_m)$ has $m$ distinct zeros in $(a,b)$.
\end{theorem}

\begin{proof}[Proof of Theorem \ref{key}]
First, note that the moments of $G$ in the $Q_k$ basis are
\[
\int_a^b Q_k(x)G(x,t)dx = \delta_k h_kR_k(t),
\] 
whence the measure associated with the generating function is regular by the hypotheses on $R_k$.
Now suppose that the polynomial $\sum_{k=0}^nq_kR_k(t)$ has zeros $\{t_1,\ldots,t_n\}$ which are all real, distinct, and in the interval $(a,b)$.  We claim the polynomial $p_m(x)=\sum_{k=0}^m\frac{q_k}{\delta_k h_k}Q_k(x)$ is the $m$-th polynomial in the biorthogonal polynomial system with respect to $G(x,t)dx$ and the sequence $\{t_1,\ldots t_n\ldots\}$.  Therefore, $p_n$ has only real zeros by Theorem \ref{real-zeros} ($G$ is positive modulo a sign change and has the interpolation property since it is SSR).  Indeed, expanding $p_m$ in the basis $Q_k$,
\[
\int_a^b p_m(x)G(x,t)dx = \sum_{k=0}^m\left(\frac{q_k}{\delta_kh_k}\right) \delta_k h_kt^k = \sum_{k=0}^mq_kR_k(t).
\] 
Since $\int_a^b p_m(x)G(x,t)dx = 0$ for $t=t_1,\ldots t_n$, the claim follows (for distinct zeros in $(a,b)$).  

If the zeros of $\sum_{k=0}^nq_kR_k(t)$ are real and in $(c,d)$, but are not distinct, consider a perturbation $\sum_{k=0}^nq'_kt^k$ which has only simple real zeros in $(c,d)$.  The corresponding zeros of $\sum_{k=0}^n\frac{q'_k}{\delta_kh_k}Q_k(x)$ are in $(a,b)$, and letting $q'_k\to q_k$ the zeros must lie in $[a,b]$ by continuity.  Then by Lemma \ref{open-sets}, the zeros of $\sum_{k=0}^n\frac{q_k}{\delta_kh_k}Q_k(x)$ must lie in $(a,b)$.
\end{proof}

%

\end{document}